\pdfoutput=1   % make sure to use pdflatex, for the arxiv

\documentclass[a4paper, 12pt]{article}
\usepackage[a4paper, hmargin= 3cm, vmargin=2cm]{geometry}

\usepackage{graphicx}
\usepackage{amsmath, amssymb, amsfonts, amscd, amsthm, mathrsfs, mathtools, dynkin-diagrams}
\usepackage[all]{xy}
\usepackage[colorlinks, linkcolor= blue, citecolor= red]{hyperref}
\usepackage[capitalise]{cleveref}
\usepackage{pdfsync}
\usepackage{color}

\usepackage[T1]{fontenc}
\usepackage[utf8]{inputenc}

\usepackage{lmodern}

\usepackage{titlesec} 

\titleformat{\section}{\centering\sc}{\S\arabic{section}.}{.5em}{}[]

\titleformat{\subsection}{\sl\bfseries}{}{0em}{}[]

\renewcommand{\phi}{\varphi} 
\newcommand{\CC}{\mathcal{C}}

\newcommand{\f}{\mathbb{F}}

\newcommand{\Fq}{\mathbb{F}_q}

\newcommand{\Height}{\mathrm{H}}

\newcommand{\SL}{\mathrm{SL}}
\newcommand{\SU}{\mathrm{SU}}

\newcommand{\Sp}{\mathrm{Sp}}

\newcommand{\PSL}{\mathrm{PSL}}
\newcommand{\PSU}{\mathrm{PSU}}

\newcommand{\RC}{\operatorname{RC}}

 \newcommand{\widesim}{
  \mathrel{{\scalebox{1.5}[1]{$\sim$}}}
}

\newcommand{\reducesize}[2]{%
  \mathbin{% This will be a binary math symbol (in terms of spacing around it)
    \ooalign{% Overlay a number of symbols
      \raisebox% Adjust vertical positioning of <object>
       {.4ex}% Move it down relative to current font     
          {$#1\widesim$}% First symbol (\sim in correct math style)
      \cr % Move to next symbol
      \hidewidth% Move symbol to right (~\hfill)
      \raisebox% Adjust vertical positioning of <object>
        {-.6ex}% Move it down relative to current font
        {\scalebox% Change the "font size"
          {.75}% to 50% of current font size
          {$#1#2$}% <object> in current math style
        }% \raisebox
      \hidewidth% Move symbol to left (~\hfill)
    }% \ooalign
  }% \mathbin
}% \reducesize

\newcommand{\nreducesize}[2]{%
  \mathbin{% This will be a binary math symbol (in terms of spacing around it)
    \ooalign{% Overlay a number of symbols
      \raisebox% Adjust vertical positioning of <object>
       {.4ex}% Move it down relative to current font     
          {$#1\not\widesim$}% First symbol (\sim in correct math style)
      \cr % Move to next symbol
      \hidewidth% Move symbol to right (~\hfill)
      \raisebox% Adjust vertical positioning of <object>
        {-.6ex}% Move it down relative to current font
        {\scalebox% Change the "font size"
          {.75}% to 50% of current font size
          {$#1#2$}% <object> in current math style
        }% \raisebox
      \hidewidth% Move symbol to left (~\hfill)
    }% \ooalign
  }% \mathbin
}% \reducesize

\newcommand{\stb}[1]{\mathpalette\reducesize{#1}}
\newcommand{\nstb}[1]{\mathpalette\nreducesize{#1}}

\usepackage{afterpage}

\newenvironment{verticallycentered}
               {\topskip0pt\vspace*{\fill}}
               {\vspace*{\fill}\clearpage}%
\usepackage{listingsutf8}

\lstnewenvironment{pythoncode}   % tailored for python. 
 {\lstset{
     inputencoding=utf8,
     breaklines=true,
     frame= lrtb,
     framerule= 0.5pt,
     rulecolor=\color{black!15},
    language=Python,
    backgroundcolor= \color{white},
    showstringspaces=false,
    formfeed=\newpage,
    tabsize=4,
    basicstyle=\small\ttfamily,
    keywordstyle=\color{blue}\bfseries,
    identifierstyle=,
    commentstyle=\color{gray}\itshape,
    stringstyle=\color{brown},
    morekeywords={lambda, True, False, None},
    literate=
    	{é}{{\'e}}{1}%
    	{è}{{\`e}}{1}%
    	{à}{{\`a}}{1}%
    	{â}{{\^a}}{1}%%%
    	{ç}{{\c{c}}}{1}%
    	{œ}{{\oe}}{1}%
    	{ù}{{\`u}}{1}%
    	{É}{{\'E}}{1}%
    	{È}{{\`E}}{1}%
    	{À}{{\`A}}{1}%
    	{Ç}{{\c{C}}}{1}%
    	{Œ}{{\OE}}{1}%
    	{Ê}{{\^E}}{1}%
    	{ê}{{\^e}}{1}%
    	{î}{{\^i}}{1}%
    	{ï}{{\"i}}{1}%%%
    	{ô}{{\^o}}{1}%
    	{û}{{\^u}}{1}%
}}
{} 

\lstnewenvironment{gapcode}   % modified from previous one for GAP, not great
 {\lstset{
     inputencoding=utf8,
     breaklines=true,
     frame= lrtb,
     framerule= 0.5pt,
     rulecolor=\color{black!15},
    language=Python,
    backgroundcolor= \color{white},
    showstringspaces=false,
    formfeed=\newpage,
    tabsize=4,
    basicstyle=\small\ttfamily,
    keywordstyle=\color{blue}\bfseries,
    identifierstyle=,
    commentstyle=\color{gray}\itshape,
    stringstyle=\color{brown},
    morekeywords={do, od, fi},
    literate=
    	{é}{{\'e}}{1}%
    	{è}{{\`e}}{1}%
    	{à}{{\`a}}{1}%
    	{â}{{\^a}}{1}%%%
    	{ç}{{\c{c}}}{1}%
    	{œ}{{\oe}}{1}%
    	{ù}{{\`u}}{1}%
    	{É}{{\'E}}{1}%
    	{È}{{\`E}}{1}%
    	{À}{{\`A}}{1}%
    	{Ç}{{\c{C}}}{1}%
    	{Œ}{{\OE}}{1}%
    	{Ê}{{\^E}}{1}%
    	{ê}{{\^e}}{1}%
    	{î}{{\^i}}{1}%
    	{ï}{{\"i}}{1}%%%
    	{ô}{{\^o}}{1}%
    	{û}{{\^u}}{1}%
}}
{}

\usepackage{tikz}
\usetikzlibrary{arrows}
\tikzstyle{sommet}=[circle,draw, scale=0.5]
\tikzstyle{infosommet}=[scale=0.75]
\tikzstyle{infoarete}=[midway, above, scale=0.85]

\newtheoremstyle{pedro}{}{}{\itshape}{}{\bfseries}{.}{ }{\thmname{#1}\thmnumber{ #2}\thmnote{ (#3)}}

\newtheoremstyle{pedrodef}{}{}{}{}{\bfseries}{.}{ }{\thmname{#1}\thmnumber{ #2}\thmnote{ (#3)}}

\newtheoremstyle{pedroimage}{}{}{\itshape\footnotesize}{}{\itshape\footnotesize}{.}{ }{\thmname{#1}\thmnumber{ #2}\thmnote{ (#3)}}

\theoremstyle{pedroimage}

\theoremstyle{pedro}
\newtheorem{lem}{Lemma}[section]

\newtheorem{thm}[lem]{Theorem}
\newtheorem*{thm2}{Theorem}

\newtheorem{prop}[lem]{Proposition}

\theoremstyle{remark}

\theoremstyle{pedrodef}

\newtheorem{ex}[lem]{Example}

\title{The binary actions of simple groups with a single conjugacy class of involutions}

\author{Nick Gill \& Pierre Guillot }

\date{}

\numberwithin{equation}{section}

\begin{document}

\maketitle

\begin{abstract} 

We continue our investigation of binary actions of simple groups. In this paper, we demonstrate a connection between the graph $\Gamma(\mathcal{C})$ based on the conjugacy class $\CC$ of the group $G$, which was introduced in our previous work, and the notion of a \emph{strongly embedded subgroup} of $G$. We exploit this connection to prove a result concerning the binary actions of finite simple groups that contain a single conjugacy class of involutions.
\end{abstract}

%\tableofcontents

\section{Introduction}

The {\em relational complexity} of a permutation group is, in principle at least, a formidable tool for the organization and classification of these objects. The simplest permutation groups, by this measure, are the {\em binary} ones, which have received intense scrutiny in recent times \cite{cherlin2, wiscons, gls_binary}. 

In \cite{npg_ppg},  we have started the study of the binary actions of simple groups. We refer to {\em loc.\ cit}.\ for more motivation and background. We have introduced the graph $\Gamma(\CC)$, where $\CC$ is a conjugacy class of the group $G$: its vertices are elements of $\mathcal{C}$, and two vertices $g,h\in \mathcal{C}$ are connected in $\Gamma(\mathcal{C})$ if $g$ and $h$ commute and either $gh^{-1}$ or $hg^{-1}$ is in $\mathcal{C}$. We have shown how the study of the connected components of $\Gamma(\CC)$ sheds considerable light on the possible binary actions for $G$. 
 
Our main result in the present paper concerns a simple group $G$ with a single conjugacy class $\CC$ whose elements are involutions. In this setting, there is a connection, thanks to Aschbacher \cite{asch2} between the connectivity properties of $\Gamma(\CC)$ and the presence in $G$ of a \emph{strongly embedded subgroup} (i.e. a proper subgroup $H$, of even order, such that $|H\cap H^g|$ is odd for all $g\in G\setminus H$).

We exploit the fact that the simple groups that contain a strongly embedded subgroup are all known, to prove the following. 

\begin{thm}\label{t: se_strong}
Suppose that $G$ is a simple group that contains a single conjugacy class of involutions, let $H$ be a proper subgroup of $G$ of even order and let $\Omega$ be the set of right cosets of $H$ in $G$. Then the action of $G$ on $\Omega$ is binary if and only if we are in one of the following situations:
\begin{enumerate}
% \item $H$ is of odd order;
      \item $G=\SL_2(2^a)$ with $a\geq 2$ and $H$ is a Sylow $2$-subgroup of $G$;
      \item $G={^2\!B_2}(2^{2a+1})$ with $a\geq 1$ and $H$ is the centre of a Sylow $2$-subgroup of $G$;
      \item $G=\PSU_3(2^a)$ with $a\geq 2$ and $H$ is the centre of a Sylow $2$-subgroup of $G$.
\end{enumerate}
\end{thm}

% Theorem \ref{thm-binary-actions-An} has left out the group $A_5 \cong \PSL_2(4) \cong \PSL_2(5)$.

For instance, for any prime power $q$, the group $\PSL_2(q)$ has a single conjugacy class of involutions, and Theorem~\ref{t: se_strong} thus provides us with a lot of information about this group and its possible binary actions. To give an example of a complete classification for a given group, the second part of the paper classifies all transitive, binary actions of $\PSL_2(q)$, with $q$ arbitrary and no restriction on the stabilizers. This also completes the picture for alternating groups $A_n$ which were considered in \cite{npg_ppg} for $n \ge 6$, since $A_5 \cong \PSL_2(4) \cong \PSL_2(5)$.

\begin{thm}\label{t: psl2 binary transitive}
Let $G=\PSL_2(q)$ act on $\Omega$, the set of right cosets of a subgroup $H<G$ and suppose $q\neq 5$. The action is binary if and only if one of the following occurs:
\begin{enumerate}
      \item $H=\{1\}$ and the action is regular;
      \item $q=2^a$ for some integer $a$ and $H$ is a Sylow $2$-subgroup of $G$;
      \item $q=2^a$ for some odd integer $a$ and $|H|=3$.
\end{enumerate}
\end{thm}

It is also possible, and much easier, to do the same for the Suzuki groups. At the end of the paper, we prove:

\begin{thm}
      Let $G={^2\!B_2}(2^{2a+1})$, with $a\geq 1$, act on $\Omega$, the set of right cosets of a subgroup $H<G$. The action is binary if and only if one of the following occurs:
      \begin{enumerate}
            \item $H=\{1\}$ and the action is regular;
            \item $H$ is the centre of a Sylow $2$-subgroup of $G$.
      \end{enumerate}
      \end{thm}

Thus we continue to see, after our initial observations with the alternating groups in \cite{npg_ppg}, that binary actions of simple groups appear to be rare. A third paper, with M. Liebeck, will fully describe the graph $\Gamma(\CC)$ when $\CC$ is a conjugacy class of involutions in a simple group of Lie type of characteristic $2$. The answer is that these graphs are ``often'' connected, again preventing the existence of binary actions in many cases, and highlighting the exceptions.
\section{Background on relational complexity}\label{s: background}

In this section we gather some basic definitions and results which will be needed in the sequel. See \cite{npg_ppg} for proofs and a more leisurely exposition.

\subsection{Complexity, height \& some basic criteria}

All groups mentioned in this paper are finite, and all group actions are on finite sets. We consider a group $G$ acting on a finite set $\Omega$ (on the right) and we begin by defining the {\em relational complexity}, $\RC(G, \Omega)$, of the action, which is an integer greater than $1$. 

% There are at least two equivalent definitions of $\RC(G,\Omega)$. The first definition involves {\em homogeneous relational structures} on $\Omega$, and while it is useful in particular for motivation, we shall work exclusively with a different definition in terms of the action of $G$ on tuples. More information on the first definition, as well as a wealth of results on relational complexity, can be found in \cite{gls_binary}.  

Let $I, J \in \Omega^n$ be $n$-tuples of elements of $\Omega$, for some $n \ge 1$, written $I=(I_1, \ldots, I_n)$ and $J= (J_1, \ldots, J_n)$. For $r \le n$, we say that $I$ and $J$ are {\em $r$-related}, and we write $I \stb{r} J$, when for each choice of indices $1 \le k_1 < k_2 < \cdots < k_r \le n$, there exists $g \in G$ such that $I_{k_i}^g = J_{k_i}$ for all $i$. 

Now the {\em relational complexity} of the action of $G$ on $\Omega$, written $\RC(G, \Omega)$, is the smallest integer $k \ge 2$ such that whenever $n \ge k$ and $I, J \in \Omega^n$ are $k$-related, then $I$ and $J$ are $n$-related (or in other words there exists $g \in G$ with $I^g = J$). One can show that such an integer always exists and indeed, it is always less than $|\Omega|$ (or it is $2$ if $|\Omega|=1$ or $2$), see \cite{gls_binary}.  

When $\RC(G, \Omega)=2$, we say that the action is {\em binary}.

% We are about to state a basic criterion for binariness that will be of constant use. Before we do this however, it is best to establish the next lemma, revealing the symmetry of a certain situation.

% \begin{lem} \label{lem-symmetry-h1h2h3}
% Let $G$ be a group. For $i=1, 2, 3$, let $H_i$ be a subgroup of $G$, and let $h_i \in H_i$. Assume that $h_1 h_2 h_3 = 1$. Then the following conditions are equivalent: \begin{enumerate}
% \item there exist $h_2' \in H_2 \cap H_1$ and $h'_3 \in H_3 \cap H_1$ such that $h_1 h_2' h_3' = 1$;
% \item there exist $h_1' \in H_1 \cap H_2$ and $h'_3 \in H_3 \cap H_2$ such that $h_1' h_2 h_3' = 1$;
% \item there exist $h_1' \in H_1 \cap H_3$ and $h_2' \in H_2 \cap H_3$ such that $h_1' h_2' h_3= 1$.
% \end{enumerate}
% \end{lem}

% When we have elements with $h_1 h_2 h_3=1$ as above, and when the equivalent conditions of the lemma fail to hold, we think of the triple $(h_1, h_2, h_3)$ as ``minimal'' or ``optimal'', in the sense that it cannot be ``improved'' to a triple with all three elements taken from the same subgroup. As we shall see presently, the presence of such an optimal triple, with the groups $H_i$ conjugates of a single subgroup $H$, is an obstruction to the binariness of the action of $G$ on $(G:H)$.

\begin{lem}[basic criteria]\cite[Lemma~2.4]{npg_ppg}\label{lem-criterion-2-not-implies-3}
    Let~$G$ act on~$\Omega $. The following conditions are equivalent:

    \begin{enumerate}
    \item There exist $I,J \in \Omega^3$ such that $I \stb{2} J$ but $I \nstb{3}J$.
    \item There are points~$\alpha_i \in \Omega$ with stabilizers~$H_i$, and elements~$h_i \in H_i$, for~$i=1,2,3$, satisfying :
  
    \begin{enumerate}
    \item $h_1 h_2 h_3 = 1$,
    \item there do NOT exist~$h_2' \in H_2 \cap H_1$, $h_3' \in H_3 \cap H_1$ with~$h_1 h_2' h_3' = 1$.
    \end{enumerate}
  
%(See also Lemma \ref{lem-symmetry-h1h2h3} above.)

\item There are points~$\alpha_i \in \Omega$ with stabilizers~$H_i$, for~$i=1,2,3$, such that $H_1 \cap H_2 \cdot H_3$ is not included in $H_1 \cap (H_1 \cap H_2) \cdot H_3$.

    \end{enumerate}
    
    In particular, when these conditions hold, the action of $G$ on $\Omega$ is not binary.
  \end{lem}

There is a particular case in which these conditions can be replaced by easier ones. To state this, let us assume that $G$ again acts on $\Omega$, and write $G_\Lambda$ for the pointwise stabilizer of $\Lambda \subset \Omega$. Define a subset $\Lambda$ of $\Omega$ to be {\em independent} when $\Lambda' < \Lambda \implies G_\Lambda < G_{\Lambda'}$. The {\em height} of the action, denoted by $\Height(G, \Omega)$, is the maximal size of an independent set. An important example is provided by the action of $G$ on the cosets of $H$, a {\em trivial intersection subgroup}, or {\em TI-subgroup} of $G$, meaning that $H \cap H^g$ is either $H$ or $\{ 1 \}$, for $g \in G$. In this situation, the height is $1$ if $H$ is normal in $G$ and $2$ otherwise. 

The connection between height and relational complexity is given by the following lemma from \cite{gls_binary}:

\begin{lem} \label{lem-ineq-RC-height}
Let $G$ act on $\Omega$. Then 
\[ \RC(G, \Omega) \le \Height(G, \Omega) + 1 \, .  \tag*{$\square$}\]
\end{lem}

One can use this to improve Lemma \ref{lem-criterion-2-not-implies-3}:

\begin{lem} \label{lem-RC-height-2}
Suppose that $G$ acts on $\Omega$, the set of cosets of a TI-subgroup $H$. Then the conditions of Lemma \ref{lem-criterion-2-not-implies-3} are also equivalent to: \begin{enumerate}
\item[4.] The action is not binary.
\item[5.] There are conjugates $H_1, H_2$ and $H_3$ of $H$ which are distinct and satisfy $H_1 \cap H_2 \cdot H_3 \ne \{ 1 \}$.
\end{enumerate}
\end{lem}
\begin{proof}
As just pointed out, the height of the action is $2$, and so the complexity is $2$ or $3$, by the previous lemma.

Now, the conditions of Lemma \ref{lem-criterion-2-not-implies-3} imply (4). Conversely, if (4) holds, then the complexity is $3$ by Lemma~\ref{lem-ineq-RC-height}. This implies (1). Indeed, if we were to suppose that (1) fails, then any two $n$-tuples which are $2$-related would be $3$-related, and thus they would be $n$-related also as the complexity is $3$, but then by definition this would show that the action is binary.

Now assume (5). By assumption we have $H_1 \cap H_2 = \{ 1 \}$ as well as $H_1 \cap H_3 = \{ 1 \}$, and we see that $H_1 \cap H_2 \cdot H_3$ is not included in $H_1 \cap (H_1 \cap H_2) \cdot H_3 = \{ 1 \}$, so we have condition (3).

Finally, assume (3). To prove that we have (5), we only need to check that the $H_i$'s are distinct. If we had $H_2 = H_3$, then $H_1 \cap H_2 \cdot H_3 = H_1 \cap H_2$ would be included in $H_1 \cap (H_1 \cap H_2) \cdot H_3$, which is absurd. If we had $H_1 = H_3$, then $H_1 \cap H_2 \cdot H_3 = H_1$ and $H_1 \cap (H_1 \cap H_2) \cdot H_1 = H_1$, another contradiction. Assuming $H_1 = H_2$ similarly leads to an absurd conclusion.
\end{proof}

\subsection{\texorpdfstring{The graph $\Gamma(\CC)$}{The graph Gamma(C)}}

Let $\CC$ be a conjugacy class of $p$-elements in $G$. We define $\Gamma(\CC)$ to be the graph whose vertices are the elements of $\CC$, and with an edge between $x, y \in \CC$ if and only if \begin{enumerate}
\item $x$ and $y$ commute, 
\item either $xy^{-1} \in \CC$ or $yx^{-1} \in \CC$.
\end{enumerate}

\begin{ex}
While studying the group $\PSU_3(q)$, we shall encounter an example where $\Gamma(\CC)$ is isomorphic to the set of nonzero cubes in $\f_q$ (where $q$ is even in this instance), with an edge between $x$ and $y$ if and only if $x+y$ is also a cube: see Lemma \ref{lem-cubes}. A little later, as we turn our attention to $\PSL_2(q)$, we shall be in a situation where a certain subgraph of $\Gamma(\CC)$ is similarly defined with squares instead of cubes: see Lemma \ref{l: p elements}. It is easy to produce more examples of this kind, close cousins of the familiar Payley graphs, by considering a semidirect product $G= \f_q \rtimes \f_q^\times$, with the action of $\f_q^\times$ of your choice.
\end{ex}

Some vocabulary in order to state the main result. When $g \in \CC$, the {\em component group of $g$ in $\Gamma(\CC)$} is the subgroup of $G$ generated by all the elements in the connected component of $\Gamma(\CC)$ containing $g$. The {\em component groups of $\Gamma(\CC)$} are the various groups thus obtained by varying $g$ ; they are all conjugate in $G$. 

The next theorem is Corollary 2.19 from \cite{npg_ppg}. It asserts that, in the case of transitive, binary actions, each stabilizer must contain a component group.

\begin{thm} \label{coro-connected-comp}
Let $G$ act on the set of cosets of a subgroup $H$, and assume that the action is binary. Let $p$ be a prime dividing $|H|$, and let $\CC$ be a conjugacy class of $p$-elements of $G$ of maximal $p$-fixity. Then for any $g \in \CC \cap H$, the component group of $g$ in $\Gamma(\CC)$ is contained in $H$.
\end{thm}

If particular, suppose that $\Gamma(\CC)$ is connected and that $G$ is simple. Then $H$ must contain the subgroup generated by $\CC$, which is normal, and we conclude in this situation that $H=G$.

\begin{ex}\label{ex: edgeless}
While most uses of the theorem in this paper will exploit the fact that $\Gamma(\CC)$ has sometimes large connected components, thus preventing the stabilizers in certain binary actions from being too small, and thereby giving an obstruction to the very existence of nontrivial binary actions, here is an easy situation where we can do the opposite.  Suppose that $\CC$ is a conjugacy class of involutions in a group $G$, let $h\in \CC$ and $H=\langle h\rangle$, and suppose that $\Gamma(\CC)$, the graph described above, does not have a single edge. Then it is an easy consequence of Lemma~\ref{lem-RC-height-2} that in this case the action of $G$ on the set of right cosets of $H$ in $G$ is binary.

To see that such actions exist, let $G=\PSL_n(q)$ with $nq$ odd and take $\mathcal{C}$ to be the conjugacy class of involutions that are the projective image of a matrix in $\SL_n(q)$ with a $(-1)$-eigenspace of dimension $n-1$ and a $1$-eigenspace of dimension $1$. It is easy to see that if $g,h\in \CC$ then $gh$ has a $1$-eigenspace of dimension at least $n-2$. 

Thus, for $n\geq 5$, the graph $\Gamma(\CC)$ is edgeless and the group $\PSL_n(q)$ has a transitive binary action with a point-stabilizer of order $2$.

The problem of classifying all simple groups $G$ with a conjugacy class $\CC$ of involutions for which $\Gamma(\CC)$ is edgeless is a tantalising one!
\end{ex}

\subsection{A few lemmas}

We shall need four lemmas from \cite{gls_binary}.  

\begin{lem}\cite[Lemma~1.6.2]{gls_binary}\label{lem-intermediate-binary}
Let $H < B < G$. Then $\RC(G, (G:H)) \ge \RC(B, (B:H))$. In particular, if the action of $G$ on $(G:H)$ is binary, so is the action of $B$ on $(B:H)$. 
\end{lem}

\begin{lem}\cite[Lemma~1.7.1]{gls_binary}\label{l: point stabilizer}
 Let $G$ be transitive on $\Omega$ and let $M$ be a point-stabilizer in this action. Let $\Lambda$ be a non-trivial orbit of $M$. Then
 \[
  \RC(G,\Omega) \geq \RC(M,\Lambda).
 \]
\end{lem}

\begin{lem}\cite[Lemma~1.8.1]{gls_binary}\label{l: frobenius}
Let $G$ be a Frobenius permutation group on $\Omega$ (that is, $G$ acts transitively on $\Omega$, $G_\omega\ne 1$ for every $\omega\in \Omega$ and $G_{\omega\omega'}=1$ for every $\omega,\omega'\in \Omega$ with $\omega\ne \omega'$). If $G$ is binary, then a Frobenius complement has order equal to $2$.
\end{lem}

In the next lemma we write $G_\Lambda$ for the setwise stabilizer of $\Lambda$, $G_{(\Lambda)}$ for the pointwise stabilizer of $\Lambda$ and $G^\Lambda=G_\Lambda/G_{(\Lambda)}$ for the permutation group induced by $G$ on $\Lambda$.

\begin{lem}\cite[Lemma~1.7.8]{gls_binary}\label{l: beautiful}
 Let $G$ act on $\Omega$ and suppose that $\Lambda\subset \Omega$. Suppose, moreover, that $G_\Lambda$, the setwise stabilizer of $\Lambda$ in $G$, acts 2-transitively on $\Lambda$. If the action of $G$ on $\Omega$ is binary, then the action of $G^\Lambda$ is the full symmetric group on $\Lambda$.
\end{lem}

\section{Groups with a single class of involutions}\label{s: one-class}

In this section we prove Theorem~\ref{t: se_strong}.

\subsection{Strongly embedded subgroups}

Consider a finite group $G$ containing a single class of involutions $\CC$. In this case $\Gamma(\CC)$ is the \emph{involution commuting graph} of $G$. 

We need a definition: a proper subgroup $N$ of $G$ is called \emph{strongly embedded} if $|N\cap N^g|$ is odd for every $g\in G\setminus N$. Now the following proposition follows immediately from the main result of Aschbacher in \cite{asch2}: 

\begin{prop}\label{p: ses}
Suppose that $G$ is a finite group containing a single class of involutions $\CC$, let $X$ be a connected component of $\Gamma(\CC)$ and let $N$ be the normalizer in $G$ of $X$. Either $X=\CC$ or else $N\cap \langle \CC\rangle$ is strongly embedded in $\langle \CC\rangle$.
\end{prop}

Assume from now on that $G$ is simple, so that $\langle \CC \rangle = G$. The alternative of the proposition is thus: either $\Gamma(\CC)$ has a single component group which is all of $G$, or $N$ is strongly embedded in $G$.

However, the latter case is severely restricted. Indeed, the Bender-Suzuki theorem \cite{bender, suzuki, suzuki2}, together with Glauberman's $Z^*$-theorem \cite{glauberman}, yields the following:

\begin{prop}\label{p: ses2}
 Suppose that $G$ is a simple group containing a strongly embedded subgroup. Then $G$ is isomorphic to $\PSL_2(q)$, ${^2B_2}(q)$ or $\PSU_3(q)$ where $q=2^a$ and $a\geq 2$.
\end{prop}

It is now easy to deduce:

\begin{prop}\label{p: se_weak}
Suppose that $G$ is a simple group acting on the set of cosets of a subgroup $H$ of even order. Assume that the action is binary, and that $H \ne G$. If $G$ has a single conjugacy class of involutions, then $G$ is isomorphic to $\PSL_2(q)$, ${^2\!B_2}(q)$ or $\PSU_3(q)$ where, in all three cases, $q=2^a$ and $a\geq 2$. Moreover, in all three cases, $H$ contains the centre of a Sylow $2$-subgroup of $G$.
\end{prop}

\begin{proof}
Theorem \ref{coro-connected-comp} and Propositions~\ref{p: ses} and \ref{p: ses2} yield the first part. The only thing that needs to be proved is that, when $G$ is in one of the three listed families of simple groups, $H$ must contain the centre of a Sylow $2$-subgroup of $G$. This follows from the fact that the centre of a Sylow $2$-subgroup of $G$ is elementary-abelian and so its non-trivial elements are all connected to each other in $\Gamma(\CC)$ where $\CC$ is the unique class of involutions in $G$.
\end{proof}

Let us list some simple groups with a single class of involutions. First, the alternating groups and the sporadics (using \cite{atlas}):
\begin{equation}\label{e: sporadics}
 A_5, A_6, A_7, M_{11}, M_{22}, M_{23}, J_1, J_3, McL, O'N, Ly, Th.
\end{equation}
Second, groups of Lie type (using \cite{gls3}):
\begin{equation}\label{e: lie}
 \PSL_2(q), {^2B_2}(q), {^2G_2}(q), \PSU_3(q), G_2(q_o), {^3D_4}(q_o).
\end{equation}
(Here we write $q_o$ to mean that $q$ must be odd.) Note that we are not claiming that these lists are exhaustive. For instance \cite{atlas} tells us that $\PSU_4(3)$ has a single class of involutions. 

For any such group $G$, except the three families of exceptions listed in the proposition, we see that $G$ does not have a single transitive, binary action with proper stabilizers of even order. This result contributes to the general impression that binary actions of simple groups are rare.

At this point, however, we do meet actual examples of nontrivial binary actions, taking $G$ as imposed by the proposition, and $H$ to be the centre of a Sylow 2-subgroup.  To see this, let us turn to the proof of Theorem \ref{t: se_strong}

\subsection{The proof}

For convenience, we repeat the statement which was given in the introduction.

\begin{thm2}
Suppose that $G$ is a simple group that contains a single conjugacy class of involutions, let $H$ be a proper subgroup of $G$ of even order and let $\Omega$ be the set of right cosets of $H$ in $G$. Then the action of $G$ on $\Omega$ is binary if and only if we are in one of the following situations:
\begin{enumerate}
% \item $H$ is of odd order;
        \item $G=\PSL_2(2^a)$ with $a\geq 2$ and $H$ is a Sylow $2$-subgroup of $G$;
        \item $G={^2\!B_2}(2^{2a+1})$ with $a\geq 1$ and $H$ is the centre of a Sylow $2$-subgroup of $G$;
        \item $G=\PSU_3(2^a)$ with $a\geq 2$ and $H$ is the centre of a Sylow $2$-subgroup of $G$.
\end{enumerate}
\end{thm2}

% \begin{proof}

% Now suppose $|K| > 1$, so we may pick $k \in K$ with $k \ne 1$. For any $a \in A$, we have $T^{ka} \cap T^a = T_0$ as just observed, so $T_0 \subset T^a$, and thus it is impossible to find $b \in A$ with $T^b \cap T^a = \{ 1 \}$. We conclude that $A$ comprises just one coset of $K$, or in other words, $K$ has index $2$ in $N$.

% It also follows that $T_0 = T$, so that the action of $T$ on $K$ is trivial. Indeed, if it were not the case we would pick a prime $p$ dividing the order of $T/T_0$, and note that $|K| = 1$ mod $p$ while $|A|=0$ mod $p$, which is at odds with the fact that $|A| = |K|$.

% Further, pick $a \in A$ and $t \in T$. There exists a unique $k = k_t\in K$ such that $a^t = k_ta$. For $s \in T$, we have $a^{ts} = (a^t)^s = k_t^s a^s = k_t a^s = k_t k_s a$. Thus $t \mapsto k_t$ is a homomorphism $T \longrightarrow K$ which is injective since the action is free.

% Finally, as $a^2 \in K$ we see that $(a^2)^t = a^2$ for any $t \in T$, which is readily rewritten as $k_t^a = k_t^{-1}$. As a result, for $s \in T$ the identity $k_{st}^a = (k_sk_t)^a = k_s^a k_t^a$ implies $k_s^{-1} k_t^{-1} = k_t^{-1} k_s^{-1}$, and we see that $T$ is abelian.
% \end{proof}

We embark on the proof, which will occupy the rest of this section.  From Proposition~\ref{p: se_weak}, we see that we may conduct the proof assuming that  $G$ is isomorphic to one of $\PSL_2(q)$, ${^2\!B_2}(q)$ or $\PSU_3(q)$ where, in all three cases, $q=2^a$ and $a\geq 2$. Moreover, in all three cases, $H$ contains the centre of a Sylow $2$-subgroup of $G$.
 
\bigskip

\noindent {\em Some binary actions.} We start by proving that if $H$ is as small as possible, i.e. $H$ is equal to the centre of a Sylow $2$-subgroup of $G$, then the action of $G$ on $\Omega$ is binary (note that if $G=\PSL_2(q)$, then a Sylow $2$-subgroup of $G$ is equal to its own centre). Note first that, in all cases, $G$ acts 2-transitively by conjugation on the set of conjugates of $H$; moreover, distinct conjugates of $H$ intersect trivially and so we use Lemma~\ref{lem-RC-height-2} to confirm that the action of $G$ on $\Omega$ is binary. We must show that if $H_1, H_2$ and $H_3$ are distinct conjugates of $H$, then $H_1.H_2\cap H_3=\{1\}$.

In the case where $G=\PSL_2(q)$, then, since $q$ is even, we have $G=\SL_2(q)$ and we use 2-transitivity to take $H_1$ (resp. $H_2$) as the set of strictly upper-triangular (resp. lower-triangular) matrices. Now observe that, for $a,b\in\Fq$,
\[
 \begin{pmatrix}
  1 & a \\ & 1
 \end{pmatrix}\begin{pmatrix}1 & \\ b & 1 \end{pmatrix} = 
 \begin{pmatrix}
  1+ab & a \\ b& 1
 \end{pmatrix}.
\]
An element $g\in\SL_2(q)$ has order dividing $2$ if and only if ${\rm trace}(g)=0$; we conclude that the element on the right hand side of the above equation has order dividing $2$ if and only if $a$ or $b$ is equal to $0$. Using the fact that distinct conjugates of $H$ intersect trivially, we conclude, as required, that $H_1.H_2\cap H_3=\{1\}$ whenever $H_1, H_2$ and $H_3$ are distinct conjugates of $H$.

When $G=\PSU_3(q)$ we use the fact that if $H_1$ and $H_2$ are distinct conjugates of $H$, then $\langle H_1, H_2\rangle\cong \SL_2(q)$ and we use the calculation for $G=\SL_2(q)$ to conclude that $H_1.H_2\cap H_3=\{1\}$ once again, as required.

Finally, when $G={^2\!B_2}(q)$, we use \cite[\S13]{suzuki} to write down two conjugates of $H$, using the natural embedding of $H$ in $\Sp_4(q)$. Recall that $q=2^{2a+1}$ for some positive integer $a$ and we set $\theta$ to be the automorphism of $\Fq$ given by $\alpha \mapsto \alpha^{2^{a+1}}$. Now the conjugates of $H$ can be written as:
\begin{align*}
 H_1 &= \left\{h_1=\begin{pmatrix}
         1 & & & \\ & 1 & & \\ \beta_1 & & 1 & \\ \beta_1^\theta & \beta_1 & & 1
        \end{pmatrix} \mid \beta_1 \in \mathbb{F}_{2^{2a+1}}\right\}; \\
 H_2 &= \left\{h_2=\begin{pmatrix}
         1 & & \beta_2 & \beta_2^\theta \\ & 1 & & \beta_2 \\  & & 1 & \\  & & & 1
        \end{pmatrix} \mid \beta_2 \in \mathbb{F}_{2^{2a+1}}\right\}.
\end{align*}

If we multiply the two matrices, $h_1$ and $h_2$, given in $H_1$ and $H_2$, then we obtain
\[
h_1.h_2= \begin{pmatrix}
  1 & 0 & \beta_2 & \beta_2^\theta \\ 0 & 1 & 0 & \beta_2 \\ \beta_1 & 0 & 1+\beta_1\beta_2 & \beta_1\beta_2^\theta \\ \beta_1^\theta & \beta_1 & \beta_1^\theta\beta_2 & 1+\beta_1\beta_2+\beta_1^\theta\beta_2^\theta
 \end{pmatrix}.
\]
One can check directly that $(h_1h_2)^2=1$ if and only if $\beta_1=0$ or $\beta_2=0$. This implies that $h_1.h_2$ lies in a conjugate of $H$ if and only if $h_1=1$ or $h_2=1$. This confirms that $H_1.H_2\cap H_3=\{1\}$ for any distinct conjugate, $H_3$, of $H$, as required.

\bigskip

\noindent {\em No other examples.}  For the remainder we assume that $H$ is a proper subgroup of $G$ that properly contains the centre of $P$, a Sylow $2$-subgroup of $G$. Our job is to show that, in this case, the action of $G$ on $\Omega$ is not binary.

First, note that in all cases, $N_G(Z(P))=P\rtimes T$, where $T$ is a group of size $q-1$ (for $G=\SL_2(q)$ or $G={^2\!B_2}(q)$) or $(q^2-1)/\gcd(q+1,3)$ (for $G=\PSU_3(q)$). 

Assume, first, that $G=\SL_2(q)$ or $G={^2\!B_2}(q)$. Then $T$ acts on $P$ fixed-point-freely and so $N_G(Z(P))$ is a Frobenius group with Frobenius kernel $P$ and a Frobenius complement $T$. What is more, since distinct conjugates of $Z(P)$ generate $G$ in this case, we conclude that $H$ normalizes $Z(P)$. If $H\neq O_2(H)$, the largest normal 2-subgroup of $H$, then this means that $H$, too, is a Frobenius group with kernel $O_2(H)$ and a Frobenius complement $T_H$ isomorphic to a subgroup of $T$; in particular $|T_H|\geq 3$.

Now, distinct conjugates of $N_G(Z(P))$ intersect in a conjugate of $T$ and it follows that there exist distinct conjugates of $H$ that intersect in a conjugate of $T_H$. This implies that there is an orbit of $H$ on which $H$ acts as a Frobenius group with stabilizer $T_H$. Now Lemma~\ref{l: frobenius} implies that the action of $H$ on this orbit is not binary and Lemma~\ref{l: point stabilizer} implies that the action of $G$ on $\Omega$ is not binary, as required.

We conclude that if $G=\SL_2(q)$ or $G={^2\!B_2}(q)$, then $H=O_2(H)$. If $G=\SL_2(q)$ this implies that $H=Z(P)$, a contradiction and this case is done. Thus assume that $G={^2\!B_2}(q)$ and $H=O_2(H)$; in particular $Z(P)<H\leq P$, for $P$ some Sylow $2$-subgroup of $G$. 

We need an explicit matrix representation of a Sylow $2$-subgroup of $G$. We define $\theta:\mathbb{F}_{2^{2a+1}} \to \mathbb{F}_{2^{2a+1}}, \zeta \mapsto \zeta^{2^{a+1}}$ and set
\begin{align*}
 U_2&:=\left\{\begin{pmatrix}
              1 & & & \\ \alpha & 1 & & \\ \alpha^{1+\theta}+\beta & \alpha^\theta & 1 & \\ \alpha^{2+\theta}+\alpha\beta+\beta^\theta & \beta & \alpha & 1
             \end{pmatrix} \mid \alpha,\beta\in\mathbb{F}_{2^{2a+1}}
        \right\}; \\
\tau&:= \begin{pmatrix}
         &&&1 \\ &&1& \\&1&& \\1&&&
        \end{pmatrix}.        
\end{align*}
Referring to \cite[\S13]{suzuki}, we find that we can define $G$ to equal $\langle U_2, U_2^\tau\rangle$ and both $U_2$ and $U_2^\tau$ are then Sylow $2$-subgroups of $G$. If we identify $P$ with $U_2$, then $Z$, the centre of $U_2$, becomes the set of elements for which $\alpha=0$.

Now observe that
\[
 \begin{pmatrix}
  1 & & & \\1 & 1 & & \\ 1 & 1 & 1 & \\ 1 & 0 & 1 & 1
 \end{pmatrix}\begin{pmatrix}
       1 & 1 & 0 & 1 \\ & 1 & 1 & 1 \\ & & 1 & 1 \\ & & & 1
      \end{pmatrix}=\begin{pmatrix}
              1 & 1 & 0 & 1 \\ 1 & 0 & 1 & 0 \\ 1 & 0 & 0 & 1 \\ 1 & 1 & 1 & 1
             \end{pmatrix}.
\]
We identify $h_2$ (resp. $h_3$, $h_1$) with the first (resp. second, third) of these matrices. Note that $h_2$ and $h_3$ are of order $4$ while $h_1$ is of order $2$. 

The element $h_2$ is in $U_2$ (set $\alpha=1$ and $\beta=0$ in the description above). Since distinct Sylow $2$-subgroups of $G$ intersect trivially, $U_2$ is the unique Sylow $2$-subgroup containing $h_2$. Likewise we write $U_1$ (resp. $U_3$) for the unique Sylow $2$-subgroup containing $h_1$ (resp. $h_3$). Note that $U_2$ is a set of lower-triangular matrices and $U_3=U_2^\tau$ is a set of upper-triangular matrices and, since $h_3$ is neither upper- nor lower-triangular, we can see that $U_1, U_2, U_3$ are distinct.

We know that $G$ contains a unique class of involutions and we also use the easy fact that $G$ contains a unique conjugacy class of cyclic subgroups of order $4$. Thus we can assume that $H_1$ (resp. $H_2$, $H_3$) contains $h_1$ (resp. $h_2$, $h_3$). But now, since $U_1, U_2$ and $U_3$ are distinct, the same is true of $H_1$, $H_2$ and $H_3$. But $h_1\in H_1\cap H_2\cdot H_3$ and so Lemma~\ref{lem-RC-height-2} yields the result.

\bigskip

\noindent {\em The more difficult case of $\PSU_3(q)$.}  We are left with the possibility that $G=\PSU_3(q)$. We set $q\geq 4$, since $G$ is simple. We let $H$ and $\Omega$ be as in the statement of Theorem~\ref{t: se_strong} with the added assumption that $H$ properly contains $Z(P)$, the centre of $P$, a Sylow $2$-subgroup of $G$; we assume that the action of $G$ on $\Omega$ is binary and will argue so as to obtain a contradiction.

We start with general information about $\PSU_3(q)$. Let $\{e_1, w, f_1\}$ be a hyperbolic basis for $V=(\mathbb{F}_{q^2})^3$ with respect to a non-degenerate Hermitian form and define three subgroups of the associated special isometry group, $\SU_3(q)$:
\begin{align*}
 P^\dagger &= \left\{\begin{pmatrix}
                  1 & a_1 & a_2 \\ & 1 & a_1^q \\ & & 1
                 \end{pmatrix} \mid a_1, a_2\in \mathbb{F}_{q^2} \textrm{ and }a_2+a_2^q+a_1^{q+1}=0 \right\} \\
T^\dagger &= \left\{\begin{pmatrix}
                  t & & \\ & t^{q-1} &  \\ & & t^{-q}
                 \end{pmatrix} \mid t \in \mathbb{F}_{q^2} \textrm{ and }t \neq 0 \right\} \\
L^\dagger  &=\left\{\begin{pmatrix}
                  a & & b \\ & 1 &  \\  c & & d
                 \end{pmatrix} \mid a,b,c,d\in \Fq \textrm{ and } ad-bc=1 \right\}
\end{align*}
Now we can take $P$ to be the projective image of $P^\dagger$ with $Z(P)$ being the projective image of the subgroup for which $a_1=0$ and $a_2\in \Fq$. Then $N_G(P)$ is the projective image of $P^\dagger \rtimes T^\dagger$. 

Another fact about $G$ that will be useful is that for any nonzero $c \in \f_{q^2}^\times$, the action of $G$ on pairs of isotropic vectors $(v,w)$ such that $\langle v, w \rangle = c$ is transitive. This follows easily from Witt's theorem.

We must consider two possibilities:

\smallskip
\noindent {\em Case (1) : $H$ contains two distinct conjugates of $Z(P)$.} Since, as we saw above, $G$ acts 2-transitively on the set of conjugates of $Z(P)$, there is a unique conjugacy class of subgroups of $G$ generated by 2 distinct conjugates of $Z(P)$; this class includes $L$, the projective image of $L^\dagger$, a subgroup of $G$ isomorphic to $\SL_2(q)$.

Thus, in this first case, we assume that $H$ contains $L$. Consulting \cite[Tables~8.5 and 8.6]{bhr} we see that $H$ normalizes $L$ and $|H:L|$ divides $(q+1)/\gcd(3,q+1)$.
Now define $Q^\dagger$ to be the subgroup of $P^\dagger$ obtained by restricting $a_1$ to $\Fq$ in the definition above. Then $Q^\dagger$ is a group of order $q^2$ and we take $Q$ to be its projective image in $G$. Define $\Lambda$ to be the subset of $\Omega$ whose elements are subsets of $H.Q$. Then $\Lambda$ is a set of size $q$, on which $Q$ acts transitively.

Let $R^\dagger = T^\dagger \cap L^\dagger$, so elements of $R^\dagger$ are diagonal matrices in $\SU_3(q)$ whose entries are in $\Fq$; let $R$ be the projective image of $R^\dagger$ in $G$ and notice that $R$ is a cyclic subgroup of $H$ of order $q-1$. Now $R$ normalizes $Q$ and hence, for all $q_1\in Q$ and $r\in R$, there exists $q_2\in Q$ such that
\[
 (Hq_1)r = (Hr)q_2 = Hq_2.
\]

Thus $R$ acts on $\Lambda$. Indeed, since $R< H$, $R$ acts on $\Lambda\setminus\{H\}$, a set of size $q-1$. We claim that, for $q\in Q$ and $r\in R$, $(Hq)r=Hq$ if and only if $Hq=H$ or $r=1$. To see this, observe first that an element of $H$ is the projective image of a matrix of form
\[
 h^\dagger = \begin{pmatrix}
  as & & bs \\ & 1/s^2 & \\ cs & & ds
 \end{pmatrix}
\]
where $a,b,c,d\in\Fq$, $s\in \mathbb{F}_{q^2}$, $ad-bc=1$ and $s^{q+1}=1$. Fixing this $h^\dagger$, we find that for $q^\dagger\in Q^\dagger$, we have
\[
 h^\dagger q^\dagger = \begin{pmatrix}
                        as & aa_1 s & aa_2s+bs \\ 0 & 1/s^2 & a_1/s^2 \\ cs & ca_1s & ca_2s+ds
                       \end{pmatrix}
                       \]
for some $a_1\in\Fq$, $a_2\in \mathbb{F}_{q^2}$ with $a_1^2+a_2+a_2^q=0$. Notice that if we vary the element $h^\dagger$ but fix $q^\dagger$, then the quotient of the non-zero elements of the middle row of $h^\dagger q^\dagger$ is fixed and equal to $a_1$. In particular all elements which project onto $Hq$ have this quotient equal to $a_1$.

Finally, fixing this $q^\dagger$ and $h^\dagger$, we find that for $r^\dagger \in R^\dagger$
\[
 h^\dagger q^\dagger r^\dagger= \begin{pmatrix}
                        ras & aa_1 s & (aa_2s+bs)/r \\ 0 & 1/s^2 & a_1/(rs^2) \\ crs & ca_1s & (ca_2s+ds)/r
                       \end{pmatrix}
                       \]
for some $r\in \Fq^*$. Notice that for this element, the quotient of the non-zero elements of the middle row is equal to $a_1/r$. We conclude that $(Hq)r=Hq$ if and only if $a_1/r=a_1$, which is if and only if $a_1=0$ or $r=1$. But if $a_1=0$, then $Hq=H$ and the claim follows.

Since $|R|=|\Lambda\setminus\{H\}|=q-1$, the claim implies that $R$ acts transitively on $\Lambda\setminus\{H\}$ and we conclude that the group $Q\rtimes R$ acts 2-transitively on $\Lambda$.

%HERE

% Thus $R$ acts on $\Lambda$. Indeed, since $R\in H$, $R$ acts on $\Lambda\setminus\{H\}
% $, a set of size $q-1$. What is more, direct matrix calculation confirms that, for $q\in Q$ and $r\in R$, $(Hq)r=Hq$ if and only if $Hq=H$ or $r=1$. Since $|R|=|\Lambda\setminus\{H\}|=q-1$, this implies that $R$ acts transitively on $\Lambda\setminus\{H\}$ and we conclude that the group $Q\rtimes R$ acts 2-transitively on $\Lambda$. 

Then Lemma~\ref{l: beautiful} implies that, since the action of $G$ on $\Omega$ is binary, $G$ must contain a section isomorphic to $A_{q-1}$, the alternating group on $q-1$ letters; but now \cite[Proposition~5.3.7]{kl} implies that this is impossible for $q\geq 8$. Thus $G=\PSU_3(4)$ and \cite{gap} confirms that here too we have a contradiction, as required.

\smallskip 

\noindent {\em Case (2): $H$ contains a unique conjugate of $Z(P)$.} In particular $Z(P)<H\leq B:= N_G(Z(P))$. The group $B$ is the Borel subgroup of $G$ and, using its well-known structure, we see that $H=Q\rtimes T_1$ where $Z(P)\leq Q\leq P$ and $T_1$ is a (possibly trivial) group of order dividing $(q^2-1)/\gcd(q+1,3)$. We shall first prove that we can reduce to the case $Q=P$, and then that we can assume that the order of $T_1$ divides $q+1$. At this point, the action of $G$ on $(G:H)$ can be understood more concretely, and we will conclude by direct arguments.

\medskip

(a) To begin with, we shall consider the action of $B$ on $(B:H)$. By  Lemma~\ref{lem-intermediate-binary}, it is also binary. We shall use this to prove, using the apparatus of component groups again, that $H$ must contain $P$.

Note that $H$ contains $Z(P)$ which is normal in $B$ and hence lies in the kernel of the action of $B$ on $(B:H)$. Define $\overline{B}:= B/Z(P)$ and $\overline{H}:=H/Z(P)$ and observe that $\overline{B}= \overline{P}\rtimes \overline{T}$ where $\overline{P}$ is elementary-abelian of order $q^2$ and $\overline{T}$ is a cyclic group of order $(q^2-1)/\gcd(q+1,3)$; similarly we can take $\overline{H}= \overline{Q}\rtimes \overline{T_1}$ where $1 \leq \overline{Q}\leq \overline{P}$ and $1 \leq \overline{T_1}\leq \overline{T}$. Also, the conjugation action of $\overline{T}$ on $\overline{P}$ can be described as follows: if we use the entry $a_1$ in the description of $P^\dagger$ to identify $\overline{P}$ with $\f_{q^2}$, then the action of the element which is the image of $\operatorname{diag}(t, t^{q-1}, t^{-q})$ is by multiplication by $t^{q-2}$. By assumption the action of $\overline{B}$ on $(\overline{B}:\overline{H})$ is binary.

Suppose, first, that $\overline{Q}=\{1\}$. In this case $\overline{H}=\overline{T_1}$ is non-trivial by assumption and $\overline{T_1}< \overline{B_1}:=\overline{P}\rtimes \overline{T_1} \leq \overline{B}$. But now the action of $\overline{B_1}$ on $(\overline{B_1}:\overline{T_1})$ is a Frobenius action and hence, by Lemma~\ref{l: frobenius}, is not binary (note that $|\overline{T_1}|$ is odd). Now Lemma~\ref{lem-intermediate-binary} implies that the action of $\overline{B}$ on $(\overline{B}:\overline{H})$ is not binary and so the same is true of the action of $B$ on $(B:H)$, a contradiction.

Suppose, instead, that $\overline{Q}\neq\{1\}$. In particular $\overline{Q}$ contains an involution $g$. We claim that the component group of $g$ is $\overline{P}$, so that $\overline{Q} = \overline{P}$ by Theorem \ref{coro-connected-comp}. The claim is very easy when $q$ is of the form $q= 2^{2a}$, for in this case, the non-identity elements of $\overline{P}= \f_{q^2}$ form the unique conjugacy class $\CC$ of involutions in $\overline{B}$, and it is a direct consequence of the definitions that $\Gamma(\CC)$ is connected (it is in fact a complete graph). The component group of $g$ is generated by $\CC$ and so it is clearly $\overline{P}$.

When $q= 2^{2a+1}$, there are three conjugacy classes of involutions in $\overline{B}$, and the class $\CC$ of $1 \in \f_{q^2} = \overline{P}$ comprises all the cubes in $\f_{q^2}^\times$ (it is also the image of the map $t \mapsto t^{q-2}$ on $\f_{q^2}^\times$ considered above). Lemma \ref{lem-cubes} below shows that the component group of $1$ is then $\overline{P}$. If we see $\f_{q^2}^\times$ as the disjoint union of the three conjugacy classes $\CC, \CC'$ and $\CC''$, then multiplication by any $x\in \f_{q^2}^\times$ provides an automorphism of the disjoint union of the graphs $\Gamma(\CC)$, $\Gamma(\CC')$ and $\Gamma(\CC'')$, and it follows easily that the conclusions of Lemma \ref{lem-cubes} can be extended to the other conjugacy classes.

We have established that $\overline{H}=\overline{P}\rtimes \overline{T_1}$. Here $\overline{H}$ is normal in $\overline{B}$ and so the action of $\overline{B}$ on $(\overline{B}:\overline{H})$ is binary, and a contradiction will not be found. Thus we return to considering the action of $G$ on $(G:H)$ directly. As announced, however, now we know that $H$ contains $P$, and in fact that $H=P\rtimes T_1$. 

\medskip

(b) Next we claim that the order of $T_1$ must, in fact, divide $(q+1)/\gcd(q+1,3)$. To see this, we note that it is easy to find $H_1$, a conjugate of $H$, such that $H\cap H_1=T_1$ and now Lemma~\ref{l: point stabilizer} implies that the action of $H$ on $(H:T_1)$ is binary. If $|T_1|>1$ and $|T_1|$ divides $q-1$, then the action of $H$ on $(H:T_1)$ is Frobenius and Lemma~\ref{l: frobenius} gives a contradiction. If $|T_1|>1$ and $|T_1|$ does not divide $q-1$, then $T_1$ has a subgroup $T_0$ of order $\gcd (|T_1|, q+1/\gcd(q+1,3))$ satisfying the conditions of Lemma \ref{lem-pseudo-frobenius} below with $N=P$ and $K=Z(P)$. Thus $T_1=T_0$ and the claim follows.

\medskip

(c) Let $T_1^\dagger$ denote the preimage of $T_1$ in $\SU_3(q)$. The group $T_1^\dagger$ can be identified with a subgroup of $\f_{q^2}^\times$ in various ways, for example {\em via} $\operatorname{diag}(t, t^{q-1}, t^{-q})\mapsto t$ ; in any case, there is a unique subgroup $T_1' \subset \f_{q^2}^\times$ having the same order as $T_1^\dagger$.

Assume, first, that $T_1^\dagger$ is non-trivial and pick $x \in T_1'$ with $x \ne 1$. We write $X$ for the set of equivalence classes of isotropic vectors in $V$ under the relation $v \sim t v $ for $t \in T_1'$. Then $G$ acts on $X$, transitively, and with stabilizers conjugate to $P \rtimes T_1$, so we work with this action and show that it admits tuples that are $2$-related but not $3$-related (condition (1) from Lemma \ref{lem-criterion-2-not-implies-3}). 

For this, pick $y \in \f_{q^2}$ such that $x + x^q = y^{q+1}$, so that $v= (x,y, 1)$ and $v'=(1, y, x)$ are isotropic. The tuples will be $([e_1], [f_1], [v])$ and $([e_1], [f_1], [v'])$, where the notation $[-]$ refers to the equivalence class. One sees easily that the tuples are $2$-related, using the transitivity of the action of $G$ on pairs of isotropic vectors with a fixed hermitian product, and the fact that $x$ is taken from $T_1'$. Suppose now that $g \in G$ were to take one triple to the other; as $g$ fixes the lines spanned by $e_1$ and $f_1$, respectively, it is of the form $\operatorname{diag}(t, t^{q-1}, t^{-q})$, and the condition is thus that there should exist $s \in T_1'$ such that \[ (tx, t^{q-1} y, t^{-q}) = (s, sy, sx) \, . \] In particular $t^{-q} = sx$, and so raising to the power $q+1$ (remembering that $x$ and $s$ are in $T_1'$ and so have order dividing $q+1$), we get $t^{-q} = t$ so $t=sx$. Comparing with $tx = s$ we get $x^2 = 1$ so $x=1$, a contradiction.

We are left with the case when $T_1^\dagger$ is trivial. This means, first, that $q= 2^{2a}$ and, second, that $H=P$. In this situation $\PSU(3,q) = \SU(3,q)$, so we can consider the action of $G$ on the set of isotropic vectors; this set may be identified with $\Omega$, with its action by $G$, as the stabilizer of an isotropic vector is $H$. We shall write down two triples of isotropic vectors which are $2$-related but not $3$-related, showing that the action cannot be binary. For a vector $v=(x, y, z)$, writing $\langle -,- \rangle$ for the Hermitian product, we have $\langle v,v \rangle = xz^q + x^qz + yy^q$, and $\langle e_1, v \rangle = z$ while $\langle f_1, v \rangle = x$. Now we pick any $x$ with $x + x^q = 1$ and $y \ne 1$ such that $yy^q = 1$, and we define $v=(x, y, 1)$, $v' = (x, 1, 1)$. The vectors $v, v'$ are isotropic, and the triples $(e_1, f_1, v)$ and $(e_1, f_1, v')$ are $2$-related, as follows again from the fact that $G$ acts transitively on pairs of isotropic vectors with constant hermitian product. However, these triples are not $3$-related, as an element of $G$ fixing $e_1$ and $f_1$ must be the identity.

\bigskip

This concludes the study of the group $\PSU(3,q)$ and the proof of Theorem \ref{t: se_strong}. 

\subsection{Auxiliary results}

During the above argument, we have relied on two lemmas which we present now. First we have:

\begin{lem} \label{lem-cubes}
Let $q = 2^{2a+1}$, and let $C$ be the set of cubes in $\f_{q^2}^\times$.
\begin{enumerate}
\item Let $\Gamma$ be the graph on $C$ in which two elements $x, y \in C$ are joined by an edge if and only if $x+y \in C$. Then $\Gamma$ is connected.
\item The additive group $\f_{q^2}$ is generated by $C$.
\end{enumerate}
\end{lem}

\begin{proof}
(1) Note that $C$ is a subgroup of $\f_{q^2}^\times$. Any $c \in C$ gives rise to a graph automorphism $\varphi_c$ of $\Gamma$ given by multiplication by $c$. When $c$ is a neighbour of $1$, we see that $\varphi_c$ must stabilize the connected component $\Delta$ containing $1$; as a result, if $c = c_1 c_2 \cdots c_s$ where each $c_i$ is a neighbour of $1$, we also have $\varphi_c(\Delta) = \Delta$, and in particular $c = \varphi_c(1) \in \Delta$. Thus it suffices to prove that $C$ is generated as a group by the neighbours of $1$, in order to conclude that $\Delta  = C$.

Any element $x \in \f_q^\times$ is a cube, and $x+1 \in \f_q^\times$ is another cube (assuming $x \ne 1$), so this $x$ is a neighbour of $1$. We see that the subgroup $C_0$ generated by these neighbours contains $\f_q^\times$, so that its order is divisible by $q-1$. We now show that any element $x \in \f_{q^2}$ satisfying $x^m = 1$, where $m = (q+1)/3$, is also a neighbour of $1$ (it is a cube since $x^{(q^2 - 1)/3} = x^{(q-1)m}= 1$).  Clearly this will suffice to show that $C_0=C$.

Let $x$ be such an element. Cubing, we see that $x^{q+1} = 1$ so that $x^q = x^{-1}$. To check that $1+x$ is a cube, we show that $(1+x)^{(q^2-1)/3} = 1 $ and for this we compute: 
\[ (1+x)^{(q-1)m} = \frac{(1+x)^{qm}}{(1+x)^m} = \left( \frac{1 + x^q}{1+x} \right)^m = \left( \frac{1 + x^{-1}}{1 + x}\right)^m = \frac{1}{x^m} = 1 \, .  \]
(2) Since the elements of $\f_q$ are cubes, the subgroup generated by $C$ is really an $\f_q$-linear subspace, so that its order is $1$, $q$ or $q^2$ by linear algebra. The cardinality of $C$ is greater than $q$, so the result follows.
\end{proof}

The second, quite technical lemma is a variation on Lemma \ref{l: frobenius} about Frobenius actions. Indeed, the proof is an elaboration of an argument due to Winscons which served to prove Lemma \ref{l: frobenius} (see \cite{gls_binary}). Note that, when the hypothesis $|T_0| > 1$ is not satisfied, then the action is Frobenius.

\begin{lem} \label{lem-pseudo-frobenius}
        Let $G =N \rtimes T$ be a semi-direct product. Let $K$ be a normal subgroup of $G$ with $K \subset N$, and let $A = N \smallsetminus K$. We assume that the action of $T$ on $A$ by conjugation is free, and that the action of $T$ on $K \smallsetminus \{ 1 \}$ factors as a free action of a quotient $T/T_0$. Also, to avoid degenerate cases, we assume $|N:K| > 1$ and $|T_0| > 1$.
        
        Finally, we assume that the action of $G$ on the set of cosets of $T$ is binary. Then $T_0 = T$. Moreover, if $|N:K| > 2$, then $|T| \le 1 + 2 |K|$.
        
        \end{lem}

        \begin{proof}
        If $|N:K|=2$, and there is a prime $p$ dividing the order of $T/T_0$, we notice that $p$ divides $|A|$, while $|K| = 1$ mod $p$. This is absurd, as $|A| = |K|$ in this case, so we conclude that $T_0=T$. We continue the argument under the assumption $|N:K| > 2$.
        
        The set of cosets of $T$ in $G$ may be identified with $N$, in such a way that the action of the subgroup $N$ is by right multiplication, and the action of $T$ is by conjugation. The stabilizer of $1$ is $G_1 = T$ and the stabilizer of $x \in N$ is $G_x = T^x$.

        First we note that $T \cap T^a = \{ 1 \}$ for $a \in A$. Indeed, $T \cap T^a$ is the stabilizer of $a$ for the action of the subgroup $T$, and this is assumed to be free. By the same token, for $k \in K \smallsetminus \{ 1 \}$, we have $T\cap T^k = T_0$.

        Pick $a, b \in A$ with $a \ne b$. We claim that $T^b \cap T^a = \{ 1 \}$ when $a$ and $b$ are not in the same (right) coset of $K$ in $N$, and $T^b \cap T^a = T_0^a = T_0^b$ otherwise. To see this, note $(T^b \cap T^a)^{b^{-1}} = T \cap T^{ab^{-1}}$ and by the previous paragraph, this intersection is trivial unless $k:=ab^{-1} \in N \smallsetminus A = K$, so that $a= kb$; by the second point made in the previous paragraph, when $a=kb$, we have $T^b \cap T^a = T_0^b$. Also, elements of $T_0$ commute with elements of $K$ by assumption, so $T_0^k = T_0$ and it follows that $T_0^a = T_0^b$.
        
        Fix $a_0 \in A$. We are going to study the various sets $X_a = a_0 \cdot G_a \smallsetminus \{ a_0 \}$, where $a \in A$ and $a_0 \cdot G_a$ is the orbit of $a_0$ under $G_a$. Suppose that, for some $a,b$, we can find $a_1$ such that $a_1\in (a_0 \cdot G_a\cap a_0 \cdot G_b)\setminus a_0 \cdot G_{ab}$. Set $I=(a_0,a,b)$ and $J=(a_1,a,b)$ and observe that $I \stb{2} J$ but $I \nstb{3}J$. This contradicts the fact that the action of $G$ on $N$ is assumed to be binary. Hence we conclude that, for any $a, b \in A$ the intersection of $a_0 \cdot G_a$ and $a_0 \cdot G_b$ is $a_0 \cdot G_{ab}$. 
        
        So if $a$ and $b$ are not in the same coset of $K$, we have $X_a \cap X_b = \emptyset$. On the other hand, if $a= kb$ for some $k \in K \smallsetminus \{ 1 \}$, then $X_a \cap X_b = a_0 \cdot T_0^a \smallsetminus \{ a_0 \}$, where $T_0^a = T_0^b$. We may put $Z_a = a_0 \cdot T_0^a \smallsetminus \{ a_0 \}$, and $Z_a$ only depends on the coset of $K$ containing $a$.
        
        If we write $Y_a$ for the union of all the sets $X_{ka}$, for $k \in K$, then $Y_a$ is the disjoint union of $Z_a$ and the various $X_{ka} \smallsetminus Z_a$. When $a$ and $b$ are not in the same coset of $K$, the sets $Y_a$ and $Y_b$ are disjoint.
        
        Now when $a$ does not lie in $Ka_0$, we have $| a_0 \cdot G_a| = |G_a| = |T|$, since $G_{a_0 a} = \{ 1 \}$, so $|X_a| = |T| -1$. Also $|a_0 \cdot T_0^a | =|T_0^a| = |T_0|$, so we have $|Z_a| = |T_0| - 1$.
        
        Suppose now that $A$ is the disjoint union of the cosets $Ka_0, Ka_1, \ldots, Ka_s$. The union of the sets $Y_{a_i}$ for $1 \le i \le s$ is disjoint, and it forms a subset of $A$, so we have 
        \[ s \big[  |T_0| - 1 + |K|(|T| - |T_0|)  \big]  \le |N| - |K| \, . \tag{*} \]
        Put $e = [N:K]$, so that $s= e-2$, and put $d= [T:T_0]$. From (*) we have 
        \[ 1  - \frac{1}{|T_0| } + |K|(d-1) \le \frac{|N| - |K|}{|T_0|(e-2)} \, , \tag{**}\]
        and then using that $|T_0| \ge 2$ and some simple rearrangement, we obtain 
        \[ d \le 1 + \frac{1}{2} \left( 1 + \frac{1}{e-2} \right) \le 2 \, . \]
        When $e >3$ we conclude that $d < 2$, so that $d=1$ and $T=T_0$. It remains to examine the possibility $e=3$ and $d=2$. We plug this back in (**) and, writing $t= |T|$ and $n= |N|$, we end up with 
        \[ 1 - \frac{2}{t} + \frac{n}{3} \le \frac{4n}{3t} \, . \]
        We rework this to obtain 
        \[ t \le 4 - \frac{6}{n+3} < 4 \, . \]
        This is absurd, however, as $|T_0| \ge 2$ so that $t = |T| \ge 4$. We have proved that $d=1$ and $T_0 = T$.
        
        Going back to (*), we are left with 
        \[ |T| \le 1 + \frac{|N|}{e} \left(1 + \frac{1}{e-2}\right) \le 1 + \frac{2|N|}{e} = 1 + 2|K| \, . \]
        This concludes the proof.
        \end{proof}

\section{The binary actions of \texorpdfstring{$\PSL_2(q)$}{PSL(2,q)}} \label{sec-epilogue}
 
This section is devoted to the proof of Theorem~\ref{t: psl2 binary transitive}, which was stated in the introduction, and which is reproduced here:

\begin{thm2}
Let $G=\PSL_2(q)$ act on $\Omega$, the set of right cosets of a subgroup $H<G$ and suppose $q\neq 5$. The action is binary if and only if one of the following occurs:
\begin{enumerate}
      \item $H=\{1\}$ and the action is regular;
      \item $q=2^a$ for some integer $a$ and $H$ is a Sylow $2$-subgroup of $G$;
      \item $q=2^a$ for some odd integer $a$ and $|H|=3$.
\end{enumerate}
\end{thm2}

Since $G=\PSL_2(q)$ has a single conjugacy class of involutions, Theorem~\ref{t: se_strong} implies that we need only deal with the case where $H$ has odd order. We must prove that, under the supposition that the action of $G$ on the cosets of $H$ is binary, $H$ is trivial, or else we are in case (3) of the theorem.

A quick look at the subgroup structure of $G$ tells us that if $H$ has odd order, then there are two possibilities:
\begin{enumerate}
 \item $H$ is cyclic of order dividing $q\pm 1$;
 \item $p$ is odd and $H$ is a subgroup of a Borel subgroup of $G$.
\end{enumerate}

We will make use of a result that follows from results in \cite{garion}.

\begin{prop}\label{p: psl2 conj}
 Let $\CC$ be a conjugacy class in $\PSL_2(q)$ with $q\geq 3$. Then $\CC^2$ contains $\CC$.
\end{prop}

We start with the case where $H$ is cyclic of order dividing $q\pm 1$. We make use of the fact that in this case $H$ is a TI-subgroup of $\PSL_2(q)$.

\begin{lem}\label{l: psl cyclic sub}
Let $H$ be a non-trivial cyclic subgroup of $G=\PSL_2(q)$ with $q>2$ and $|H|$ dividing $q\pm 1$. Then there always exist $H_1, H_2, H_3$, distinct conjugates of $H$ such that $H_3 \subset H_1.H_2$, unless $q=2^a$ with $a$ odd and $|H|=3$ (in which case the $H_i$'s cannot be found). 

In particular the action of $G$ on the set of cosets of $H$ is binary if and only if $q=2^a$ with $a$ odd and $|H|=3$.
\end{lem}

\begin{proof}
Write $h$ for a generator of $H$. We use the fact that the only conjugates of $h$ that lie in $H$ are $h$ and $h^{-1}$. Now, by Proposition~\ref{p: psl2 conj}, we know that there exist $h_1, h_2$ conjugate to $h$ such that $h_1h_2 = h$. If $\langle h_1\rangle\neq \langle h_2\rangle$, then we can take $H_1=\langle h_1\rangle$, $H_2=\langle h_2\rangle$ and $H_3=\langle h\rangle$ and we have $H_3\subset H_1.H_2$ as required. 

Consider the case where $\langle h_1\rangle=\langle h_2\rangle$. In this case either $h_2=h_1$ or $h_1=h_1^{-1}$; the latter case is impossible because then $h_1h_2=1=h$ . The former case has $h_1^2$ a conjugate of $h_1$ and so $h_1^2=h_1^{-1}$ and we conclude that $|H|=3$. The rest of the proof thus assumes that $s=3$, and that $q$ is not a power of $3$.

Here are a few easy facts. An element of $g \in \SL_2(q)$ has order $3$ if and only if its trace is $-1$, if and only if its characteristic polynomial is $X^2+X+1$, if and only if it is conjugate to 
\[ g_0 = \left(\begin{array}{rr}
 -1 & 1 \\ 
 -1 & 0 
\end{array}\right) \, . \]
It follows that an element $h \in \PSL_2(q)$ has order $3$ if and only if $h= gZ$ for $g$ of trace $\pm 1$, if and only if $h$ is conjugate to $g_0 Z$. In particular, the elements of order $3$ form just one conjugacy class in $G$. % As a result, we are exactly after two elements $h_1, h_2 \in G$ of order $3$ with $h_1 \ne h_2$ and such that $h_3 = h_1 h_2$ also has order $3$ (putting $H_i = \langle h_i \rangle$ then produces the desired subgroups, and this is the only way to obtain them).

We may as well choose $h_1 = g_1 Z$ where 
\[
 g_1 = \begin{pmatrix}
  1 & -1 \\ 1 & 0
 \end{pmatrix} \, .
\]
Then put $h_2 = g_2 Z$ where 
\[
 g_2=\begin{pmatrix}
      x & y \\ z & t
     \end{pmatrix}
\]
and we assume that $x+t=\pm1$. Then observe that
\[
 g_1g_2=\begin{pmatrix}
       x-z & y-t \\ x & y
      \end{pmatrix}.
\]  
We want to know under what conditions we can choose $x,y,z,t\in\Fq$ such that
\[
 x+t=\pm1 \textrm{ and } xt-yz=1 \textrm{ and } x+y-z=\pm1.
\]
With some rearranging we obtain
\[
 y^2+(x\pm 1)y + (x^2\pm x+1)=0. \tag{$\ast$}
\]
Note that the two instances of ``$\pm$'' can be taken independently so, for any value of $x$ with $q$ odd we obtain up to 4 distinct quadratic equations in $y$.

If $q$ is even, then we need to solve
\[
 y^2+(x+1)y+(x^2+x+1)=0.
\]
We set $Y=y+1$ and $X=x+1$ and we obtain
\[
 Y^2+YX+X^2=0 \, ,
\]
which is homogeneous. Hence, either $X=0$ or else $\lambda^2+\lambda+1=0$ where $\lambda=Y/X$. The former solution corresponds to $x=1$ and we obtain $g_1=g_2$ and so $h_1=h_2$ and thus $H_1$ and $H_2$ are not distinct. The latter equation, $\lambda^2+\lambda+1=0,$ has a solution if and only if $a$ is even, exactly as announced in the statement of the lemma.

It remains to see that when $q$ is odd and not a power of $3$, we can find a solution to  $(\ast)$ for some choice of signs, which is different from the solution $x=1, y=-1$ (so $z=1, t=0$) which gives $g_1 = g_2$, and also different from the ``opposite'' solution which gives $g_1 = -g_2$ (in either case we would have $h_1 = h_2$). We choose to show that there is always a solution to $(\ast)$  with two ``plus'' signs, which of course  implies that $(x, y) \ne (1, -1)$ and $(x,y) \ne (-1,1)$.

First we try to find a solution with $x=0$, which requires us to pick $y$ with $y^2 +y + 1 = 0$. The discriminant of this equation is $-3$, and so $y$ can be found when $-3$ is a square in $\f_q$. To finish the proof, we now suppose that $-3$ is not a square in $\f_q$, and we introduce the quadratic form 
\[ Q(X, Y, Z) = Y^2 + (X+Z)Y + X^2 + XZ + Z^2 \, . \]
Any quadratic form in at least 3 variables over a finite field has a non-zero solution, as is classical (this is a consequence of the Chevalley-Warning Theorem). Let $(x, y, z) \ne (0,0,0)$ be such that $Q(x, y, z) = 0$. We claim that $z \ne 0$. Indeed, we have 
\[ Q(x, y, 0) = y^2 + xy + x^2 \, , \]
so that $Q(x, y, 0) = 0$ implies 
\[ \left(x+ \frac{1}{2}y \right)^2 + \frac{3}{4} y^2 = 0 \, , \]
so that 
\[ -3 = \frac{2^2 (x + \frac{1}{2}y)^2}{y^2} \in \f_q^{*2} \, . \]
This is at odds with our current assumption on the field $\f_q$, so $z \ne 0$. Now write 
\[ 0 = \frac{1}{z^2} Q(x, y, z) = y_1^2 + (x_1+1)y_1 + x_1^2 + x_1 + 1 \]
with $x_1= \frac{x}{z}$ and $y_1= \frac{y}{z}$. We are done.

The ``in particular'' now follows from Lemma~\ref{lem-RC-height-2}.
\end{proof}

We are left with the case where $p$ is odd and $H$ is a non-trivial odd-order subgroup of a Borel subgroup of $G$. We may assume that $H$ is not cyclic of order dividing $q\pm 1$, hence we assume that $U_0=O_p(H)$ is non-trivial. Then, by the structure of Borel subgroups of $G$, either $H=U_0$ (i.e. $H$ is a $p$-group) or $H=U_0\rtimes T_0$ where $T_0$ is cyclic of odd order dividing $q-1$.

This case is dealt with in four lemmas. We emphasise that we assume that $p$ is odd for the rest of this section. Since we are assuming that $H$ is non-trivial, we must show that the action of $G$ is not binary.

\begin{lem}\label{l: whatevs}
If $H\neq O_p(H)$, then the action of $G$ on the set of cosets of $H$ is not binary.
\end{lem}

\begin{proof}
In this case $H=U_0\rtimes T_0$ where $U_0=O_p(H)$ and $T_0$ is a non-trivial cyclic group of order at least $3$ and dividing $q-1$. Note that $T_0$ acts fixed-point-freely on $U_0$. 

Let $U$ be the unique Sylow $p$-subgroup containing $U_0$. Note that $N_G(T_0)$ is dihedral of order $q-1$; what is more $N_G(T_0)$ has an index $2$ subgroup $T$ that is cyclic of order $\frac{q-1}{2}$, that contains $T_0$ and that satisfies $N_G(U)=U\rtimes T$.

Now let $h\in N_G(T_0)\setminus T$ and observe that $H\cap H^h$ contains $T_0$. On the other hand, since Sylow $p$-subgroups of $G$ are TI-subgroups, $U\cap U^h=\{1\}$ and so $U_0\cap U_0^h=\{1\}$. We conclude that $H\cap H^g=T_0$.

This implies, in particular, that $H$ acts as a Frobenius group on the suborbit corresponding to $H^g$ and $T_0$ is the corresponding Frobenius complement. Thus Lemma~\ref{l: frobenius} implies that, since $|T_0|>2$, this action of $H$ is not binary and now Lemma~\ref{l: point stabilizer} yields the result we seek.
\end{proof}

\begin{lem}\label{l: consecutive squares}
 Let $q$ be an odd prime power.  Define
 \[
  C:=\{x \in \Fq \mid x \textrm{ and } x+1 \textrm{ are nonzero squares}\}.
 \]
Then
\[
 C=\begin{cases}
    \frac{q-5}{4}, \textrm{ if } q\equiv 1\pmod 4; \\
    \frac{q-3}{4}, \textrm{ if } q\equiv 3\pmod 4.
   \end{cases}
\]
\end{lem}

\begin{proof}
Suppose first that $-1$ is not a square in $\f_q$, that is, suppose $q \equiv 3$ mod 4. The map 
\[ f \colon \f_q \smallsetminus \{ 0, 1, -1 \} \longrightarrow C  \]
defined by $f(x) = \left( \frac{x - x^{-1}}{2}  \right)^2$ is well-defined, as follows from the observation that 
\[ \left( \frac{x - x^{-1}}{2}  \right)^2 +1 = \left( \frac{x + x^{-1}}{2}  \right)^2 \, , \]
and our assumption on $q$ guarantees that $x + x^{-1} \ne 0$.
However, $f$ is also onto: if $b^2 \in C$, we have $b^2 + 1 = a^2$ for some $a$, so that putting $x= a+b$ yields $x^{-1} = a - b$, and thus $b^2  = f(x)$.

Fix $y$ such that $y - y^{-1} \ne 0$, and let us count the number of $x$ such that $x - x^{-1} = y - y^{-1}$. Considering the discriminant, we conclude that there are two possible values of $x$ when $y^2 \ne -1$, and of course this is always the case given our current assumptions.  As a result, the equation $f(x)=f(y)$ has always four solutions, as it amounts to $x - x^{-1} = \pm (y - y^{-1})$, and the size of $C$ is $\frac{q-3}{4}$. 

When $q \equiv 1$ mod 4 on the other hand, writing $i$ for an element of $\f_q$ with $i^2 = -1$, we use the map $f$ defined on $\f_q \smallsetminus \{ 0,\pm 1, \pm i \}$ by the same formula as above, and argue similarly.
\end{proof}

\begin{lem}\label{l: p elements}
Let $U$ be a Sylow $p$-subgroup of $G$, let $g$ be a non-trivial element of $U$ and let $\CC$ be the conjugacy class of $g$ in $G$. If $q\neq 9$, then the component group of $g$ in $\Gamma(\CC)$ is equal to $U$.
\end{lem}

Note that if $q=9$, the the component group of $g$ in $\Gamma(\CC)$ is equal to $\langle g\rangle$. 

\begin{proof}
Recall that there are two conjugacy classes of $p$-elements in $G$ and that $|\CC\cap U|=\frac{q-1}{2}$. Consider $\Gamma_{\CC,U}$, the subgraph induced on $U\cap \CC$ in $\Gamma(\CC)$.

It is easy to see that $\Gamma_{\CC,U}$ is isomorphic to the graph $\Delta$ whose vertices are nonzero squares in $\Fq$, with two vertices, $x$ and $y$, connected if and only if $x-y$ or $y-x$ is a square in $\Fq$.

If $q\equiv 3\pmod 4$, then $-1$ is not a square in $\Fq$ and, for all $x,y\in \Fq$, exactly one of $x-y$ and $y-x$ is a square in $\Fq$. We conclude that $\Gamma_{C,U}$ is connected and so the component group of $g$ contains $\frac{q-1}{2}$ elements. We conclude that the component group of $g$ is $U$.

If $q\equiv 1\pmod 4$, then we use Lemma~\ref{l: consecutive squares}. We claim that each vertex in $\Gamma_{\CC,U}$ has valency $v=\frac{q-5}{4}$. To see this, write $x_1,\dots, x_v$ for the set of elements $x$ in $\Fq$ with the property that $x$ and $x+1$ are squares. Now, let $y$ be a square in $\Fq$ and observe that $y$ is connected to $y+yx_1, \dots, y+yx_v$ in $\Gamma_{\CC,U}$. This implies that a connected component of $\Gamma_{\CC,U}$ contains at least $\frac{q-1}{4}$ elements. Since all connected components contain the same number of elements, we either have that $\Gamma_{\CC,U}$ is connected (and we are done, as per the previous paragraph) or else $\Gamma_{\CC,U}$ has two connected components. Assume the latter. Also, let $U'$ be the component group of $g$, and assume that $U' \ne U$; we must show that $q=9$.

Suppose first that $q \le 15$. Since $q$ is an odd prime power, which is not prime (as this would imply $U'=U$), we deduce that $q=9$. So we may as well suppose that $q>15$ and look for a contradiction. Under this assumption however, we have $\frac{q-1}{4}+1>\frac{q}{5}$, so we see that the index of $U'$ in $U$ is $3$, and that $q$ is a power of $3$.

The situation is as follows: $\Gamma_{\CC,U}$ has two connected components, each containing $\frac{q-1}{4}$ vertices and, since the valency of each vertex is $\frac{q-5}{4}$, each is isomorphic to a complete graph. Write $X_1$ and $X_2$ for each component and note that $\langle X_1\rangle$ and $\langle X_2\rangle $ are both index $3$ subgroups of $U$. We see also that $\langle X_1 \rangle$ and $\langle X_2 \rangle$ together comprise at least $\frac{q-1}{2}$ elements and, since this exceeds $\frac{q}{3}$, these two subgroups generate $U$. Finally note that, since the difference of any two vertices in $X_1$ (resp. $X_2$) is a square, $\langle X_1\rangle \cup \langle X_2\rangle \subseteq X_1\cup X_2\cup\{0\}$, the set of squares in $\Fq$. But now
\[
 \frac{q}{3}+ \frac{q}{3}-\frac{q}{9} = |\langle X_1\rangle \cup \langle X_2\rangle| \leq X_1\cup X_2\cup\{0\} = \frac{q+1}{2}
\]
and we obtain that $q\leq 9$, a contradiction which finishes the proof.
\end{proof}

\begin{lem}\label{l: psl2 sylow}
If $H=O_p(H)$, then the action of $G$ on the set of cosets of $H$ is not binary.
\end{lem}
\begin{proof}
In this case, $H$ is a non-trivial $p$-group. Let $g$ be a $p$-element of maximal fixity. Lemma~\ref{l: p elements} and Theorem \ref{coro-connected-comp} imply that either $q=9$ and $|H|=3$ or else $H$ is equal to a Sylow $p$-subgroup of $G$. If $q=9$, then $G=A_6$ and the lemma follows from \cite{npg_ppg} (or can be checked directly using GAP). Thus we assume that $H$ is a Sylow $p$-subgroup of $G$.

We observe that 
\begin{equation}\label{e: h2}
\begin{pmatrix} -3 & 1 \\ -4 & 1 \end{pmatrix} = \begin{pmatrix}1 & 1 \\ 0 & 1 \end{pmatrix}\begin{pmatrix}1 & 0 \\ -4 & 1 \end{pmatrix}.
\end{equation}
We let $h_1$ (resp $h_2$, $h_3$) be the projective image in $G$ of the first (resp. second, third) matrix of \eqref{e: h2}.

Note that all matrices in \eqref{e: h2} have trace equal to $\pm 2$, hence $h_1, h_2$ and $h_3$ are all $p$-elements. Note, too, that $h_1$ is the image of an upper-triangular matrix, $h_2$ is the image of a lower-triangular matrix and $h_3$ is the image of a matrix that is neither upper nor lower-triangular. Thus $h_1, h_2, h_3$ lie in distinct Sylow $p$-subgroups of $G$ which we denote $H_1, H_2$ and $H_3$, respectively. Now, since $h_1\in H_1\cap H_2\cdot H_3$, Lemma~\ref{lem-RC-height-2} gives the result.
\end{proof}

We remark that it is easy to use a similar strategy to classify the transitive binary actions of the Suzuki groups. 

\begin{thm}\label{t: suz binary transitive}
Let $G={^2\!B_2}(2^{2a+1})$, with $a\geq 1$, act on $\Omega$, the set of right cosets of a subgroup $H<G$. The action is binary if and only if one of the following occurs:
\begin{enumerate}
      \item $H=\{1\}$ and the action is regular;
      \item $H$ is the centre of a Sylow $2$-subgroup of $G$.
\end{enumerate}
\end{thm}
\begin{proof}
The group $G$ has a unique class of involutions hence, when $|H|$ is even, the result follows from Theorem~\ref{t: se_strong}. Assume, then, that $|H|>1$ is odd. Consulting \cite{suzuki} we see that in this case $|H|$ is cyclic of order dividing $q-1$, $q-r+1$ or $q+r+1$ where $q=2^{2a+1}$ and $r=2^{a+1}$. What is more $H$ is a TI-subgroup.

Let $h$ be a generator of $H$ and let $N=N_G(H)$. Consider the action of $N$ on the set of non-trivial elements of $H$. Again referring to \cite{suzuki} it is easy to see that the orbits in this action are of size at most $4$, hence $h$ is conjugate in $G$ to at most 4 of its powers (including itself).

Let $\mathcal{C}$ be the conjugacy class of $G$ containing $h$  and recall the standard formula (see, for instance, \cite[p. 45]{isaacs}):
\[
 |\{(x,y)\in\mathcal{C} \mid xy=h\}| = \frac{|\mathcal{C}|^2}{|G|} \sum\limits_{\chi \in \mathrm{Irr}(G)} \frac{\chi(h)^2\overline{\chi(h)}}{\chi(1)}.
\]
Using the character table given in \cite[\S17]{suzuki}, together with this formula, it is easy to check that, for any fixed $H$ of odd order, the value of the expression on the right-hand side of the formula is strictly greater than $4$. Thus we can find $x, y\in \mathcal{C}$ with $x,y\not\in H$ such that $xy=H$. Now set $H_1=\langle h\rangle=H$, $H_2=\langle x\rangle$ and $H_3=\langle y\rangle$ and observe that these are distinct conjugates of $H$. By construction $h \in H_1\cap H_2\cdot H_3$   and so Lemma~\ref{lem-RC-height-2} implies the result.
\end{proof}

\bibliography{myrefs}
\bibliographystyle{amsalpha}

\end{document}